\documentclass[12pt,reqno]{amsart}
\usepackage{graphicx,indentfirst}
\usepackage{amsmath,amssymb,mathrsfs}
\usepackage{amsthm,amscd}
\usepackage{verbatim}
\usepackage{appendix}
\usepackage{enumitem,titletoc}

\usepackage[utf8]{inputenc}

\usepackage{fancyhdr}
\usepackage{amsfonts,color}
\usepackage[all]{xy}
\usepackage{tikz-cd}
\usepackage{syntonly}
\usepackage{float}
\usepackage{pgfplots}
\pgfplotsset{compat=1.18}
\usepackage{array}
\usepackage[normalem]{ulem}

\usepackage[left=2.5cm,right=2.5cm,bottom=3cm,top=3cm]{geometry}

\usepackage{tikz}
\usepackage{extarrows}
\usepackage{hyperref}
\usepackage{cancel}

\def\XXint#1#2#3{{\setbox0=\hbox{$#1{#2#3}{\int}$ }
		\vcenter{\hbox{$#2#3$ }}\kern-.6\wd0}}

\newtheorem{theorem}{Theorem}[section]

\newtheorem{lem}{Lemma}[section]
\newtheorem{prop}{Proposition}[section]
\newtheorem{cor}{Corollary}[section]

\newtheorem{remark}{Remark}[section]

\theoremstyle{definition}
\newtheorem{defn}{Definition}[section]

\newcommand{\eps}{\varepsilon}

\makeindex

\definecolor{ForestGreen}{RGB}{34,139,34}
\hypersetup{
	colorlinks=true,
	citecolor=ForestGreen,
	filecolor=ForestGreen,
	linkcolor=blue,
	urlcolor=black
}

\numberwithin{equation}{section}

\usetikzlibrary{intersections}
\usepgfplotslibrary{fillbetween}
\usetikzlibrary{patterns,patterns.meta}

\author{Tristan C. Collins}
\email{\href{mailto:tristanc@math.toronto.edu}{tristanc@math.toronto.edu}}
\address{Department of Mathematics, University of Toronto, 40 St. George Street, Toronto, ON, Canada}
\author{Benjy Firester}
\email{\href{mailto:benjyfir@mit.edu}{benjyfir@mit.edu}}
\address{Department of Mathematics, Massachusetts Institute of Technology, 77 Massachusetts Ave., Cambridge, MA, USA}
\author{Freid Tong}
\email{\href{mailto:freid.tong@utoronto.ca}{freid.tong@utoronto.ca}}
\address{Department of Mathematics, University of Toronto, 40 St. George Street, Toronto, ON, Canada}
\date{\today}

\newcommand{\p}{\partial}

\newcommand{\pt}{\partial_t}

\newcommand{\la}{\langle}
\newcommand{\rg}{\rangle}

\newcommand{\mr}[1]{{\rm #1}}

\usepackage{wrapfig,caption}
\newcommand{\cC}{\mathcal{C}}
\newcommand{\cE}{\mathcal{E}}

\newcommand{\cU}{\mathcal{U}}

\newcommand{\bA}{\mathbb{A}}
\newcommand{\bC}{\mathbb{C}}

\newcommand{\bR}{\mathbb{R}}
\newcommand{\bS}{\mathbb{S}}

\newcommand{\ol}{\overline}

\title[Homogeneous optimal transport]{Homogeneous optimal transport maps between oblique cones}

\begin{document}

\begin{abstract}
We construct homogeneous optimal transport maps for the quadratic cost between convex cones with homogeneous, possibly degenerate, densities when the cones satisfy an obliqueness condition. 
The existence of such maps plays a central role in the boundary regularity theory for optimal transport maps between convex domains. 
Our results are also relevant for the existence of complete Calabi-Yau metrics on certain quasi-projective varieties.
 
\end{abstract}
\maketitle
\vspace{-0.8cm}
\section{Introduction}
The existence, uniqueness, and regularity of optimal transport maps is an important problem with applications to differential geometry, mathematical physics, economics, probability, and computer science \cite{DPF,Evans,Villani,Villani2}. 
Seminal works of Brenier \cite{Brenier} and Gangbo-McCann \cite{Gangbo-McCann} showed the existence and uniqueness of optimal transport equations with quadratic cost as given by the gradients of solutions to an associated Monge-Amp\`ere equation. 
Caffarelli \cite{Caffarelli, Caffarelli2,Caffarelli3, Caffarelli4, Caffarelli5} developed a regularity theory for solutions of the Monge-Amp\`ere equations arising from optimal transport with quadratic cost.
Despite many advances, a complete theory of the boundary regularity remains an important open problem.

An important advance in the boundary regularity theory was made by Chen–Liu–Wang \cite{Chen-Liu-Wang}, which proved global $C^{2,\alpha}$ regularity for the Monge-Amp\`ere equation for $C^{1,1}$ convex (but not necessarily strictly convex) domains under mild assumptions on the densities. 
Recent progress building on Caffarelli’s program, and inspired by geometric regularity theory, has advanced the boundary regularity theory.
In \cite{Collins-Tong}, the first and third authors recently showed a new monotonicity formula that is constant along homogeneous optimal transport maps between cones, illustrating these as the appropriate notion of \textit{tangent cones} for the theory. 
As emphasized in \cite{Collins-Tong}, the existence and regularity of homogeneous optimal transport maps between cones are intimately related to the sharp boundary regularity for optimal transport maps.
Exploiting the monotonicity formula, \cite{Collins-Tong} established global $C^{1,1-\eps}$ regularity for optimal maps on convex domains, and global $C^{2,\alpha}$ regularity on $C^{1,\alpha}$ bounded convex domains, under mild assumptions on the densities.

The existence, uniqueness, and regularity of optimal transport maps between convex cones are central to the boundary regularity of optimal transport.  Optimal transport maps between half-spaces with possibly degenerate densities were analyzed by  Jhaveri-Savin \cite{Jhaveri-Savin}, who established an important Liouville theorem.
The combined works of Collins-Tong-Yau \cite{TristanFreidYau} and Collins-Firester \cite{TristanBenjy} constructed homogeneous optimal maps from a strict, convex cone to a half-space with possibly degenerate densities by reducing the optimal transport to a free boundary Monge-Amp\`ere equation.
As shown in \cite{TristanBenjy}, the associated class of free boundary Monge-Amp\`ere equations have interesting connections to a variety of problems in geometry and analysis. 

The goal of this paper is to establish the existence of homogeneous optimal transport maps between a general class of convex cones. 
To illustrate the relevance to the regularity theory consider an optimal transport map between polyhedral domains equipped with the Lebesgue measure. 
The boundaries of the polyhedra can be stratified based on the maximal dimension of an intersection with a supporting hyperplane.
In the ``generic case", after blowing up one expects to see a global optimal transport map between two half-spaces; such maps were analyzed by Jhaveri-Savin \cite{Jhaveri-Savin}.
For lower dimensional strata, the regularity of blow-up limits is dictated by the regularity of homogeneous solutions of the following optimal transport problem:
\[
\begin{cases}
    \det D^2\varphi &=\frac{g(x)}{g'(\nabla \varphi)}  \text{ in }\mathtt{C}\\
    \nabla \varphi(\mathtt{C}) &= \mathtt{C}'
\end{cases}
\]
where $(\mathtt{C},\mathtt{C}')$ are convex cones, and $(g(x), g'(y))$ are positive homogeneous functions on $(\mathtt C, \mathtt C')$. 
In this paper we establish the existence of such maps when the cones $\mathtt{C},\mathtt{C}'$ satisfy a strong obliqueness property; see Definition~\ref{def:Obliqueness} below.
When the source $\mathtt{C}$ splits non-trivial lines, solutions can be obtained from our result by taking products.
Our results also apply for a fairly general class of homogeneous measures.
As found in \cite{TristanBenjy}, when the target is not a strict cone, there is an apparent obstruction to the existence related to the relative degrees of degeneracy of $(g, g')$

Our results are also relevant to the existence of complete Calabi-Yau metrics on quasi-projective varieties, as initiated by Yau \cite{Yau} and Tian-Yau \cite{Tian-Yau, Tian-Yau2}.
Suppose $X$ is a Fano manifold, $\dim_{\bC}X=n$, $k<n$, and $D= D_1+\cdots +D_k$ is an anti-canonical divisor with simple normal crossings.
Suppose additionally that each divisor $D_i$ is ample.
As explained in \cite{Collins-Li}, homogeneous solutions of optimal transport problems between convex cones with degenerate densities describe the generic asymptotics (those are, tangent cones at infinity) of (putative) complete Calabi-Yau metrics.
The existence of such metrics when $k=1$ was established in foundational work of Tian-Yau \cite{Tian-Yau}, building on seminal work of Yau \cite{Yau}.
Collins-Li \cite{Collins-Li} constructed complete Calabi-Yau metrics for $k=2$ divisors satisfying the proportionality condition $D_1 \equiv_{\mathbb{Q}} D_2$ by producing an explicit, homogeneous optimal transport map from the positive orthant in $\mathbb{R}^2$ to a half-space.
In higher dimensions, the existence of homogeneous optimal transport maps between the orthant and a half-space was established in \cite{TristanFreidYau, TristanBenjy}. 
The main results of the present paper are applicable to the case when the $D_i$ are non-proportional ample divisors; see \cite[Section 2.4]{Collins-Li}. 

We now briefly outline the results obtained in this paper. 
Let $(P, \Sigma)\subset (\mathbb R^n_y, \mathbb R^n_x)$ be two convex bodies containing the origin which live in dual vector spaces, which we denote $\mathbb R^n_y$ and $\mathbb R^n_x$.
We consider the cones $(\mathtt{C}(P), \mathtt{C}(\Sigma))\subset (\mathbb R^{n+1}_y, \mathbb R^{n+1}_x)$ given by
\[
\begin{aligned}
\mathtt{C}(P)&:= \{(ty, t): y\in P, t\geq 0\}\subset \mathbb R^{n+1}_y,\\
\mathtt{C}(\Sigma) &:= \{(tx, t): x \in \Sigma, t\geq 0\} \subset \mathbb{R}^{n+1}_x.
\end{aligned}  
\]
Each cone has a dual $\mathtt{C}^\vee$ defined in the dual space as $\mathtt{C}^\vee = \{y : \la x, y \rg > 0$ for all $x \in \mathtt{C}\}$.
We equip the cones $(\mathtt{C}(P), \mathtt{C}(\Sigma))$ with densities  $(d\mu, d\nu) = (g_P(y)dy, g_\Sigma(x)dx)$ satisfying some structural assumptions; see Section~\ref{sec: prelim}.
The degrees of homogeneity of the measures will refer to the homogeneity of the functions $g_P$ and $g_\Sigma$. 

\begin{defn}\label{defn:ObliqueCone}
    We say the pair of strict cones $(\mathtt{C},\mathtt{C}')$ is \textit{strongly oblique} if
\[
     \ol{\mathtt{C}}\setminus\{0\} \subset \mr{int}((\mathtt{C}')^\vee),
\]
where $(\mathtt{C}')^\vee$ denotes the dual cone of $\mathtt{C}'$.
\end{defn}
This condition was first identified in the work of the first and third authors \cite{Collins-Tong}, where it was shown that if a pair of strongly oblique cones $(\mathtt{C}(P), \mathtt{C}(\Sigma))$ arises as affine tangent cones of an optimal transport map, then any blow-up will be a homogeneous optimal transport map between $\mathtt{C}(P)$ and $\mathtt{C}(\Sigma)$, which suggests that such homogeneous optimal transport maps should always exist.

Our first result shows that indeed under the strong obliqueness condition, homogeneous optimal transport maps always exist. 

\begin{theorem}\label{thm:StrongObliqueness}
    Suppose two strict cones $(\mathtt{C}, d\mu)$ and $(\mathtt{C}',d\nu)$ are strongly oblique with homogeneous, locally finite, doubling densities. 
    Then, there exists a homogeneous optimal transport map $\nabla \varphi$ between them for $\varphi \in C^{1,\epsilon}(\ol{\mathtt{C}(P)})$ for some $\epsilon>0$. 
\end{theorem}

We refer the reader to Section~\ref{sec: prelim} for the definition of locally finite, doubling densities. 

\begin{remark}
    Andreasson and Hultgren \cite{RolfJakob} have independently obtained a similar result to Theorem~\ref{thm:StrongObliqueness} using different techniques.
\end{remark}

Our second result considers the case when the target cone splits off $n-k$ factors of $\bR$. 
In this case, the appropriate notion of obliqueness is the following:
\begin{defn}\label{def:partialObliqueCones}
The pair $(\mathtt{C},\mathtt{C}' \times \bR^{n-k})$ is \textit{strongly partially oblique} if
\[
\ol{\mathtt{C}\cap W}\setminus \{0\}\subset \mr{relint}((\mathtt{C}')^\vee)
\]
where $W$ is the $(k+1)$-dimensional subspace containing $(\mathtt{C}')^\vee$, and $\mr{relint}(\mathtt{C}')^\vee$ is the relative interior of $(\mathtt{C}')^\vee$. 
\end{defn}

Under the assumption of strong partial obliqueness, we establish the existence of homogeneous optimal transport maps.

\begin{theorem}\label{thm:PartiallyStrongObliqueness}
    Consider two cones $(\mathtt{C}, d\mu)$ and $(\mathtt{C}' \times \bR^{n-k},d\nu)$ with $k<n$ that are strongly partially oblique, and $(d\mu,d \nu)$ are locally finite, doubling measures with homogeneous densities of respective degrees $(\alpha, \beta)$.
    Suppose that $\beta > \alpha$ and $d\mu \gtrsim d_{\p \mathtt{C}}^{1 + \alpha}dy$. 
    Then, there exists a homogeneous optimal transport map $\nabla \varphi$ between them for $\varphi \in C^{1,\epsilon}(\ol{\mathtt{C}(P)})$ for some $\epsilon>0$. 
\end{theorem}

We briefly outline the general strategy to solving this problem.
An optimal transport map $\nabla \varphi: \mathtt{C}(P) \to \mathtt{C}(\Sigma)$ is given by $\varphi$ solving the equation
\begin{equation}\label{eqn:OTbtwnCones}
\begin{cases}
 \det D^2 \varphi= \dfrac{g_P(y)}{g_\Sigma(\nabla \varphi)}
 &\text{in }\mathtt{C}(P), \\
\nabla\varphi(\mathtt{C}(P)) = \mathtt{C}(\Sigma).
\end{cases}
\end{equation}
The second condition can be equivalently stated as $\p_{n+1} \varphi = \phi_{\Sigma^\circ}(\p_1\varphi,\ldots, \p_n \varphi)$.
When $\varphi$ is homogeneous of degree $1 + \frac{n+1+\alpha}{n+1+\beta}$, the optimal transport equation reduces to the following Monge-Amp\`ere equation on the compact domain $P$, (see Lemma~\ref{lem: ansatzReduction}),
\begin{equation}\label{eqn:FreeBdyEqnP}
\begin{cases}
     \det D^2v = \dfrac{h_P(y)}{v^{n+2+\alpha}(-v^\star)^\beta h_\Sigma(\tfrac{\nabla v}{-v^\star})} & \text{in }P ,\\
    v^\star + \phi_{\Sigma^\circ}(\nabla v) = 0 & \text{on }\p P,
\end{cases}
\end{equation}
where $v^\star(y) = \la \nabla v(y),y\rg - v(y) = v^*(\nabla v)$ is pullback of the Legendre transform of $v$ by $\nabla v$. 
If we let $u = v^*$ be the Legendre transform, then the above equation can be expressed as the following free boundary problem
\begin{equation}\label{eqn:freeBdyEqnOmega}
\begin{cases}
    \det D^2 u = \dfrac{(u^\star(x))^{n+2+\alpha}(-u(x))^\beta h_{\Sigma}(\tfrac{x}{-u})}{h_P(\nabla u)} &\text{in }\Omega,\\
    u(x) + \phi_{\Sigma^\circ}(x) = 0&\text{on }\p \Omega ,\\
    \nabla u(\Omega) = P.
\end{cases}
\end{equation}
We construct an energy functional $\cE$ and show that, up to rescaling, $C^2$ solutions of~\eqref{eqn:freeBdyEqnOmega} are precisely critical points of $\cE$.
When $\Sigma$ is compact, as in Theorem~\ref{thm:StrongObliqueness}, the strong obliqueness property ensures a lower bound on the $\mathcal{E}$ functional.
On the other hand, when $\Sigma$ is not compact, as in Theorem~\ref{thm:PartiallyStrongObliqueness}, the group of translations preserving $\Sigma$ leads to non-compactness for $\mathcal{E}$.
As a result, we must normalize the functions $v$ by appropriate affine functions.  Similar arguments were used in the construction of homogeneous optimal transport maps to half-spaces \cite{TristanFreidYau, TristanBenjy}.  In comparison with the arguments in \cite{TristanFreidYau, TristanBenjy}, the main new idea introduced in the current paper is to exploit the geometry of strong obliqueness to establish suitable compactness results. New arguments are also required to deal with the general measures $h_P$ and $h_\Sigma$, resulting in a more complicated energy functional, and necessitating a new argument establishing that critical points of the energy functional satisfy the optimal transport equation.

An outline of the paper is as follows:  Section~\ref{sec: prelim} fixes notation and ideas and contains preliminary discussion of the variational problem. Section~\ref{sec:Strict} establishes energy estimates based on obliqueness and proves Theorem~\ref{thm:StrongObliqueness}. Section~\ref{sec:Non-strict} combines ideas from Section~\ref{sec:Strict}, and \cite{TristanFreidYau, TristanBenjy} to prove Theorem~\ref{thm:PartiallyStrongObliqueness}.

\smallskip
\noindent\textbf{Acknowledgments}: T.C.C.~is supported in part by NSERC Discovery grant RGPIN-2024-518857, and NSF CAREER grant DMS-1944952.
B.F.~is supported by a MathWorks fellowship.
F.T.~is supported in part by NSERC Discovery grant RGPIN-2025-06760.

\section{Preliminaries and variational framework}\label{sec: prelim}

We recall some definitions from convex geometry and establish some notation and terminology:
\begin{itemize}
    \item For a convex set $K\subset \mathbb R^n_y$, the support function $\phi_{K}:\mathbb R^n_x\to \mathbb R$ is defined by $ \phi_{K}(x) = \sup_{y\in K} \langle x,y \rangle $. 
    \item The polar dual $K^{\circ}\subset \mathbb R^{n}_x$ is a convex body in the dual space given by $K^{\circ} = \{x: \phi_K(x) \leq 1\}$. 
    \item The dual cone $\mathtt{C}^\vee$ is given by
    $\mathtt{C}^\vee = \{x : \la x, y \rg > 0$ for all $y \in \mathtt{C}\}$.
    \item For $v$ a convex function, $u = v^*$ is its Legendre dual given by $ u(x) = \sup_{y \in P}\la x ,y \rg - v(y)$.
    \item The convex indicator function $\mathbf{1}_K$ is
    \[
    \mathbf{1}_K(y) = \begin{cases} 0 & \text{ if } y \in K,\\
    +\infty & \text{ else}.
    \end{cases}
    \]
    It satisfies $\mathbf{1}_K = (\phi_K)^*$.
    \item We can express the cones over convex sets as
    \[
    \begin{split}
    \mathtt{C}(P) &= \{(\ol{y}, y_{n+1})\in \mathbb R^{n+1}_y: y_{n+1}\geq \phi_{P^{\circ}}(\ol{y})\} \text{ and}\\
    \mathtt{C}(\Sigma) &= \{(\ol{x}, x_{n+1})\in \mathbb R^{n+1}_x: x_{n+1}\geq \phi_{\Sigma^{\circ}}(\ol{x})\}.
    \end{split}
    \]
    \item For $\bR^{n+1}_x$ (resp.~$\bR^{n+1}_y$), we let $\ol{x}$ (resp.~$\ol{y}$) denote the first $n$ directions and $\tilde{x} = \tfrac{\ol{x}}{x_{n+1}}$ (resp.~$\tilde{y} = \tfrac{\ol{y}}{y_{n+1}}$) to be the rescaling to the {\em link of the cone} $\Sigma$ (resp.~$P$).  We will similarly denote $\ol{\nabla} u = (\p_1 u,\ldots, \p_n u)$ and $\tilde{\nabla}u = \tfrac{\ol{\nabla}u}{\p_{n+1}u}$. 
    \item We will use the notation $a \lesssim b$ to denote $a \leq Cb$ where $C$ is a uniform, estimable constant depending only on background data; $a\sim b$ will mean $a \lesssim b$ and $b \lesssim a$. 
    \item We say that a convex set $K$ is in John's position if its John ellipsoid is centered at the origin.
    Notably, any convex body in John's position is uniformly equivalent to $B_{|K|^{1/n}}(0)$ with comparison constants depending only on the dimension.
\end{itemize}

As a first step, we recast the strong obliqueness conditions  in Definitions~\ref{defn:ObliqueCone} and~\ref{def:partialObliqueCones} in terms of the links $\Sigma$ and $P$. 
\begin{defn}\label{def:Obliqueness}
    When $\Sigma$ is compact, we say that the pair $(P,\Sigma)$ is \textit{strongly oblique} if 
    \[
    \la \p \phi_{P}(0), \p \phi_{\Sigma}(0)\rg + 1 > 0 \iff -\overline{P} \subset \mr{int}(\Sigma^{\circ}) \iff  -\ol{\Sigma}\subset \mr{int}(P^\circ).
    \]
    When $\Sigma = \Sigma^k \times \bR^{n-k}$, we say the pair $(P,\Sigma)$ is \textit{strongly partially oblique} if
    \[
    -\ol{P \cap \bA''}\subset (\Sigma^k)^\circ = (\Sigma^k \times \bR^{n-k})^\circ \subset \bA''
    \]
    where $\bA''$ is the $k$-dimensional subspace containing the polar dual of $\Sigma^k$.
\end{defn}

\begin{lem}
    The strong obliqueness conditions on cones given in Definitions~\ref{defn:ObliqueCone} and~\ref{def:partialObliqueCones} are equivalent to the strong obliqueness conditions in Definition~\ref{def:Obliqueness} on the links.
\end{lem}
\begin{proof}
    We show that $\mathtt{C}(\Sigma)^\vee\setminus \{0\} = \mr{int}(\mathtt{C}(-\Sigma^\circ))$.
    Consider a point $(y, t) \in \mathtt{C}(\Sigma)^\vee \setminus \{0\}$ meaning that 
    \[
    \la y, x\rg + ts  > 0 \text{ for all }(x,s) \in \mathtt{C}(\Sigma)\setminus \{0\}.
    \]
    Suppose $\Sigma$ is compact.
    Rescaling by $\frac{1}{st}$ immediately implies $-\overline{P}\subset {\rm int}(\Sigma^{\circ})$.  
    Now suppose $\Sigma = \Sigma^{k}\times \mathbb{R}^{n-k}$, so $C(\Sigma)^{\vee} \subset \bA''= (\mathbb{R}^{n-k})^{\perp}$.  
    Rescaling by $\frac{1}{st}$, we obtain that strong partial obliqueness is equivalent to $-\overline{P\cap \bA''}\subset \Sigma^{\circ}$. 
\end{proof}
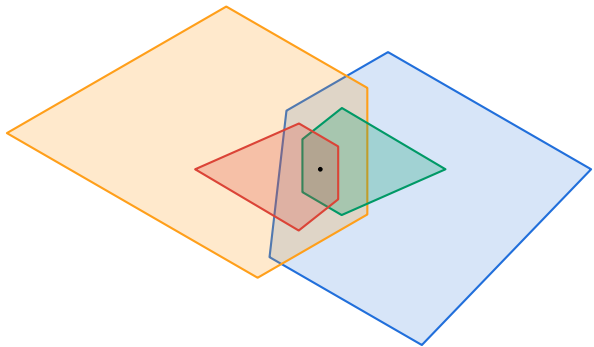
\begin{figure}
    \centering
    \begin{tikzpicture}[scale=1.5, line join=round, line cap=round]
          \definecolor{sigcolor}{RGB}{33,111,219}     
          \definecolor{polcolor}{RGB}{255,159,26}     
          \definecolor{pcolor}{RGB}{0,153,102}        
          \definecolor{mpcolor}{RGB}{219,68,55}       
        
          \filldraw[fill=sigcolor, fill opacity=0.18, draw=sigcolor, thick]
            (-0.450000,-0.779423) --
            ( 0.900000,-1.558846) --
            ( 2.400000, 0.000000) --
            ( 0.600000, 1.039230) --
            (-0.300000, 0.519615) -- cycle;

          \filldraw[fill=polcolor, fill opacity=0.22, draw=polcolor, thick]
            (-2.777778, 0.320750) --
            (-0.555556,-0.962250) --
            ( 0.416667,-0.400938) --
            ( 0.416667, 0.721688) --
            (-0.833333, 1.443376) -- cycle;

          \filldraw[fill=pcolor, fill opacity=0.25, draw=pcolor, thick]
            ( 1.109277, 0.000000) --
            ( 0.191370, 0.543377) --
            (-0.158478, 0.266366) --
            (-0.158506,-0.202480) --
            ( 0.189136,-0.406112) -- cycle;

          \filldraw[fill=mpcolor, fill opacity=0.25, draw=mpcolor, thick]
            (-1.109277,-0.000000) --
            (-0.191370,-0.543377) --
            ( 0.158478,-0.266366) --
            ( 0.158506, 0.202480) --
            (-0.189136, 0.406112) -- cycle;

          \fill (0,0) circle (0.6pt);
        
        \end{tikzpicture}
    \caption{The pair $(P,\Sigma)$ is drawn in green and blue respectively, which is strongly oblique as seen by $-P$ in red being contained in $\Sigma^\circ$ in yellow}
    \label{fig:obliquePair}
\end{figure}

We equip the cones $(\mathtt{C}(P), \mathtt{C}(\Sigma))\subset (\mathbb R^{n+1}_y, \mathbb R^{n+1}_x)$ with homogeneous measures $(d\mu, d\nu) = (g_P(y)dy, g_\Sigma(x)dx)$. 
We define the measures $h_P$ and $h_\Sigma$ to be $g_P$ and $g_\Sigma$ restricted to the links $P$ and $\Sigma$.

The measures $d\mu$ and $d\nu$ on the cones must satisfy some structural conditions, both to estimate the energy functional and to apply the Caffarelli theory.
Throughout, our measures will satisfy the following properties:
\begin{itemize}
    \item Homogeneity: $d\mu$ and $d\nu$ are homogeneous measures.
    Specifically, $g_P$ and $g_\Sigma$ are of degrees $\alpha$ and $\beta$ respectively, so 
    \[
    g_P(y) = y_{n+1}^\alpha h_P(\tilde{y})\qquad \text{and}\qquad g_\Sigma(x) = x_{n+1}^\beta h_\Sigma(\tilde{x}).
    \]
    \item Locally finite:
    the densities on the links are bounded, so
    \[\sup_{y \in P}h_P(y) < \infty \qquad \text{and}\qquad  \sup_{x \in \Sigma}h_\Sigma(x) < \infty.
    \]
    We extend $h_P$ and $h_\Sigma$ by 0 outside $P$ and $\Sigma$.
    \item Doubling: there exist constants $C_P,C_\Sigma > 0$ such that 
    \[
    \int_E h_P(y)\,dy \leq C_P \int_{\frac{1}{2}E}h_P(y)\,dy\qquad \text{and}\qquad \int_E h_\Sigma(x)\,dx \leq C_\Sigma \int_{\frac{1}{2}E}h_\Sigma(x)\,dx,
    \]
    where $E$ is any ellipsoid and $\tfrac{1}{2}E$ is rescaled from its center.
\end{itemize}
Further assumptions on the regularity of the measures will improve the regularity of solutions using \cite{Caffarelli, Caffarelli2, Caffarelli3, Caffarelli4, Caffarelli5, Chen-Liu-Wang, Collins-Tong}.

Consider the class of convex functions
\[
\cC^+(P) := \{v : P \to \bR_+ : v \text{ is convex}\}.
\]
The Legendre dual of $v \in \cC^+(P)$ is a convex function $u: \mathbb{R}^n \rightarrow \mathbb{R}$ satisfying the growth condition
\[
\phi_P - C < u < \phi_P
\]
for some $C > 0$. In particular, $\{ u < 0 \}$ is a compact convex set. 

Let $w := u + \phi_{\Sigma^\circ}$, so the free boundary is $\Omega = \{w < 0\} \subset \{ u < 0\}$.
We can now define our functionals.
Let
\[
    I(u) := \int_{\bR^n}K(x,-w(x))\chi_{\{w < 0\}}\,dx
\]
where 
\[
K(x,s) := \int_0^s g_\Sigma(x, \sigma  +\phi_{\Sigma^\circ}(x))\, d\sigma = \int_0^s (\sigma + \phi_{\Sigma^\circ}(x))^\beta h_{\Sigma}(\tfrac{x}{\sigma + \phi_{\Sigma^\circ}(x)})\, d\sigma.
\]
From this definition, one sees immediately that the variation is
\[
\delta I = -\int_{\bR^n}\dot{u}\left[(-u)^\beta h_\Sigma(\tfrac{x}{-u})\right]\chi_{\{w<0\}}\, dx .
\]
Assuming $u$ is $C^2$ and strictly convex in $\Omega= \{w<0\}$ and $\nabla u(\Omega) = P$, we can write
\[
\delta I=\int_P \dot{v} \left[(-v^\star)^\beta h_\Sigma(\tfrac{\nabla v}{-v^\star}) \det D^2 v\right]\,dy,
\]
where $v$ is the Legendre transform of $u$.

We define the functional $J$ as  
\[
J(v) := \frac{1}{n + 1 + \alpha}\int_{P}\frac{h_P(y)}{v^{n+1+\alpha}(y)}\, dy \implies \delta J = - \int_P \dot{v}\left[\frac{h_P(y)}{v^{n+2+\alpha}(y)}\right]\, dy.
\]
Finally, we define the energy functional
\[
\cE(u) := -\log I(u) + J^{-\frac{1}{n+1+\alpha}}(v),
\]
whose $C^2$ critical points solve equation~\eqref{eqn:freeBdyEqnOmega} up to a positive multiplicative constant.

\begin{remark}\label{rem:rescaling}
By homogeneity, we can always rescale a critical point $v_t = tv$ by an appropriate constant to solve equation~\eqref{eqn:FreeBdyEqnP}.
Therefore, we can ignore the multiplicative constant induced by the Euler-Lagrange equations of $\cE$.
This can be shown directly by computing the variation of $\cE(v_t)$ by homogeneity properties of both $I$ and $J$.
Although more complicated in this general density setup, similar formulas can be shown for $I(v_t)$ and $J(v_t)$ as computed in \cite[Lemma 5.2]{TristanBenjy}.
\end{remark}
\begin{lem}\label{lem:IKernelEstimate}
    For $x \in \Omega$, function $K$ satisfies
    \[
    K(x, -w(x)) \lesssim (-u(x))^{1 + \beta}.
    \]
\end{lem}
\begin{proof}
    
    We use the homogeneity of $g$ and $\sup_\Sigma h_\Sigma < \infty$, on $\Omega$ to bound
    \begin{align*}
        K(x,-w(x)) &= \int_0^{-w(x)} (\sigma + \phi_{\Sigma^\circ}(x))^\beta h(\tfrac{x}{\sigma + \phi_{\Sigma^\circ}(x)})\, d\sigma \\
        &\leq \sup_\Sigma h_\Sigma \int_0^{-w(x)} (\sigma + \phi_{\Sigma^\circ}(x))^\beta\,d\sigma \\
        &=\sup_\Sigma h_\Sigma \frac{(-w(x) + \phi_{\Sigma^\circ}(x))^{1 + \beta} - \phi_{\Sigma^\circ}(x)^{1+\beta}}{1+\beta} \\
        &\leq (\sup_\Sigma h_\Sigma) \frac{(-u(x))^{1 + \beta}}{1+\beta}
    \end{align*}
    as desired.
\end{proof}
We now show the homogeneous reduction of the optimal transport equation~\eqref{eqn:OTbtwnCones} to the free boundary equation~\eqref{eqn:FreeBdyEqnP}.
\begin{lem}\label{lem: ansatzReduction}
    Let $\tilde{y} = \left(\tfrac{y_1}{y_{n+1}},\ldots, \tfrac{y_n}{y_{n+1}}\right)$.
    If we have a separation of variables of the form $\varphi = f(y_{n+1}v(\tilde{y}))$, then $\det D^2 \varphi = \frac{(f')^{n}f''v^2}{y_{n+1}^n}\det D^2 v$.
\end{lem}
\begin{proof}
We first compute
\begin{align*}
\partial_i (y_{n+1} v(\tilde{y})) &= y_{n+1} \sum_k v_k\frac{\partial}{\partial y_i}\frac{y_k}{y_{n+1}} = v_i\\
\partial_{n+1}(y_{n+1}v(\tilde{y})) &= v + y_{n+1}\sum_k v_k\frac{\partial}{\partial y_{n+1}}\frac{y_k}{y_{n+1}} = v-\sum_k\frac{v_ky_k}{y_{n+1}}.
\end{align*}
Therefore, we can compute the first derivatives
\[
    \varphi_i =f' v_i \qquad \text{and}\qquad 
    \varphi_{n+1} = f'\left(v - \sum_k \frac{v_ky_k}{y_{n+1}}\right)
\]
and the second derivatives
\begin{align*}
    \varphi_{ij} &= f''v_i v_j + \frac{f'}{y_{n+1}}v_{ij}\\
    \varphi_{i(n+1)} &= f''v_i\left(v - \sum \frac{v_ky_k}{y_{n+1}}\right) - \frac{f'}{y_{n+1}}\sum_k v_{ik}\frac{y_k}{y_{n+1}}\\
    \varphi_{(n+1)(n+1)} &= f'' \left(v - \sum_k \frac{v_ky_k}{y_{n+1}}\right)^2 + \frac{f'}{y_{n+1}}\sum_{k,\ell} v_{k\ell}\frac{y_k}{y_{n+1}}\frac{y_\ell}{y_{n+1}},
\end{align*}
so we can express the Hessian of $\varphi$ as 
\[
D^2 \varphi = \begin{pmatrix}
    A & b \\
    b^\top & c 
\end{pmatrix},\quad A_{ij}=\varphi_{ij},\quad b_i = \varphi_{i(n+1)},\quad c = \varphi_{(n+1)(n+1)}.
\]
We therefore see that $\det D^2 \varphi = \det(A)(c - b^\top A^{-1} b)$, and we compute 
\[
\det(A) =\left(\frac{f'}{y_{n+1}}\right)^{n} \det(v_{ij})
\left( 1 +\frac{y_{n+1} f''}{f'} \sum_{i,j} v_i\, v^{ij}\, v_j \right).
\]
We can  compute $A^{-1}$ as 
\[
(A^{-1})_{ij}
=\frac{y_{n+1}}{f'}v^{ij}
-\frac{y_{n+1}^{2} f''}{f'^{2}}
\frac{\left(\sum_{p} v^{ip} v_p\right)\left(\sum_{q} v^{jq} v_q\right)}
{1+\dfrac{y_{n+1} f''}{f'}\sum_{a,b} v_a\,v^{ab}\,v_b},
\]
which shows that
\[
c - b^{\top} A^{-1} b = \frac{f'' v^{2}}{1 + \dfrac{y_{n+1} f''}{f'} \sum_{i,j} v_i v^{ij}\, v_j}.
\]
Combining the above computations gives the desired formula $\det D^2 \varphi = f'' \left(\frac{f'}{y_{n+1}}\right)^{n} v^{2}\det(v_{ij})$.
\end{proof}

Applying Lemma~\ref{lem: ansatzReduction} with $f(s) = s^{1 + \gamma}$ a homogeneous function, we can reduce the optimal transport between cones to the free boundary Monge-Amp\`ere equation defined on the link.
Define 
\[
1 + \gamma = 1 +\frac{n+1+\alpha}{n+1+\beta}.
\]
Then, $\varphi(y) = (y_{n+1}v(\tilde{y}))^{1 + \gamma}$ solves equation~\eqref{eqn:OTbtwnCones} when $v$ solves equation~\eqref{eqn:FreeBdyEqnP}.

\section{Minimizing solutions for strongly oblique cones}\label{sec:Strict}

In this section, we prove Theorem~\ref{thm:StrongObliqueness}. 
Recall from Definition~\ref{def:Obliqueness} that the pair $(P,\Sigma)$ is \textit{strongly oblique} if 
\[
    \la \p \phi_{P}(0), \p \phi_{\Sigma}(0)\rg + 1 > 0 \iff \inf_{x \in P, y \in \Sigma}\la x, y \rg + 1 > 0 \iff -\ol{P} \subset \mr{int}(\Sigma^\circ). 
\]
The strongly oblique property implies that 
\[
\phi_\cU (y) := \phi_{\Sigma^\circ}(y) - \phi_{P}(-y) > 0
\]
is a strictly positive, homogeneous function. 
It follows by compactness that there exists $C>0$ such that
\[
r|y|\leq \phi_{\cU}(y)\leq C|y|.
\]
Note that $\phi_{\cU}$ will generally not be convex. 

Assuming strong obliqueness, we will show that the functional $\cE$ is uniformly lower bounded.
Recall the notation
\[
w = u + \phi_{\Sigma^\circ}, \qquad \Omega= \{w<0\}.
\]
 Since $\nabla u (\Omega) \subset P$ by assumption, we have a priori gradient bounds on $w$ given by $\nabla w(\Omega) \subset P + \Sigma^\circ$. 

The first result from the strong obliqueness property is that $w$ is uniquely minimized at the origin, and the free boundary is uniformly equivalent to $B_{|w(0)|}(0)$. 
\begin{lem}\label{lem:RoundnessOfOmega}
    Let $\delta = |w(0)|$.
    Then $w(x) \geq w(0)$ with equality if and only if $x=0$. 
    Furthermore, there exist uniform constants $r,R$, depending only on the pair $(P, \Sigma)$, such that we have the containment
    \[
    B_{\delta r}(0) \subset \Omega \subset B_{\delta R}(0).
    \]
    Consequently, we have that $|\Omega| \sim (-u(0))^n$. 
\end{lem}
\begin{proof}
    For $\delta = |w(0)|$, we consider the values $w(t\theta)$ for any $\theta \in \bS^{n-1}$. 
    We first show that $w = u + \phi_{\Sigma^\circ}$ achieves its infimum at the origin. 
    For any $\theta \in \bS^{n-1}$ and $t \in \bR_{>0}$, we can see that for any $V \in \p u(0) \subset P$, a subgradient of $u$ at 0, we have
    \begin{equation}\label{eqn:wMinAt0}
        t\phi_{P + \Sigma^\circ}(\theta) \geq w(t\theta) - w(0) = u(t\theta) - u(0) + t\phi_{\Sigma^\circ}(\theta) \geq t( \la V, \theta\rg +\phi_{\Sigma^\circ}(\theta))
    \end{equation}
    by the convexity of $u$ and homogeneity of the convex support function.
    Minimizing the right-hand side, we know that 
    \[
    \inf_{V \in \p u(0)} \la V, \theta\rg  \geq \inf_{V' \in P} \la V', \theta\rg  = -\sup_{V' \in P} \la -V', \theta\rg = -\phi_P(-\theta)
    \]
    using the definition of $\phi_P(\theta) = \sup_{V \in P}\la V, \theta\rg$. 
    Thus, strong obliqueness implies that $\la V, \theta\rg + \phi_{\Sigma^\circ}(\theta) > 0$, and therefore
    \[
   { 0 < \phi_{\cU}(t\theta )  < w(t\theta)- w(0) <  \phi_{P+\Sigma^\circ}(t\theta)},
    \]
    from which the claim follows by setting
    \[
    \frac{1}{R} = \inf_{\theta \in \bS^{n-1}} \phi_{\cU}(\theta)\qquad \text{and}\qquad \frac{1}{r} = \sup_{\theta \in \bS^{n-1}}\phi_{P +\Sigma^\circ}(\theta).
    \]
\end{proof}
The estimates on $I$ and $J$ are most easily seen as functions of $|\inf u|$ and $\inf v$. 
The roundness of $\Omega$ and the a priori gradient estimate $|\nabla u| \leq C(P)$  will show that these quantities are uniformly equivalent.
This will reduce the estimates on $\cE$ to a single quantity $v(0)$.

\begin{lem}\label{lem:IEstimateStrict}
    For any $v \in \cC^+(P)$, we have
    \[
    u(0) \sim \inf_{x \in \Omega}u(x)
    \qquad \text{and}\qquad
    I(v) \lesssim (v(0))^{n+1 + \beta}.
    \]
\end{lem}
\begin{proof}

    We first use the strong obliqueness property to show that $u(x_0) \sim u(0)$ uniformly, where $x_0$ achieves the infimum of $u$ over $\Omega$.
    From Lemma~\ref{lem:RoundnessOfOmega}, we know that $\delta = |u(0)|$ gives a uniform control over the distance to $\p \Omega$ along any direction $\theta \in \bS^{n-1}$.
    Since $\nabla u(\Omega) \subset P$, we have a uniform gradient bound $|\nabla u| \leq \mathtt{C}(P)$.
    This Lipschitz bound tells us that for all $x \in \Omega$, we know
    \[
    u(x) \geq u(0) - C(P)|x|,
    \]
    so applying Lemma~\ref{lem:RoundnessOfOmega} which tells us that $|x| \leq |u(0)| R$, the above bound becomes
    \[
    u(x) \geq (1 + C(P)R)u(0).
    \]
    We therefore see that $u(0) \sim \inf_\Omega u$ because
    \[
    u(0) \geq \inf_{\Omega}u \geq \left(1 + C(P)R\right)u(0).
    \]
    The Legendre transform shows that $u(0) \sim \inf_{\Omega} u$ is equivalent to $v(0) \sim \inf_P v$ as claimed. 

    We now estimate $I$ using the control over the geometry of $\Omega$ from strong obliqueness.
    Applying the estimate from Lemma~\ref{lem:IKernelEstimate} pointwise on $\Omega$ shows
    \[
    I(u)= \int_\Omega K(x,-w(x))\, dx \lesssim |\Omega|\sup_{\Omega}(-u(x))^{1+\beta} = |\Omega|(-v(0))^{1+\beta}.
    \]
    Lemma~\ref{lem:RoundnessOfOmega} shows $|\Omega| \sim (-u(0))^n \sim (v(0))^n$, completing the proof.
\end{proof}

\begin{lem}\label{lem:JEstimateStrict}
    For any $v \in \cC^+(P)$, there is a uniform upper bound $J(v) \lesssim (v(0))^{-(n+1+\alpha)}$.
\end{lem}
\begin{proof}
    We can estimate
    \begin{align*}
    J(v) = \int_P \frac{h_P(y)}{v(y)^{n+1+\alpha}}\,dy \lesssim \int_P \frac{1}{(\inf_P v)^{n+1+\alpha}}\,dy \leq C \int_P \frac{1}{(v(0))^{n+1+\alpha}}\,dy 
    \end{align*}
    using the fact that $\inf_P v \sim v(0)$ from Lemma~\ref{lem:IEstimateStrict} and Legendre duality.
\end{proof}

Combining the estimates on $I$ and $J$, we can show that the sublevel sets of $\cE$ are compact and that there is a minimizing function $v$. 
\begin{cor}\label{cor:ELowerBounds}
    For any $v \in \cC^+(P)$, we have an energy lower bound
    \begin{equation}\label{eqn:ELowerBound}
    \cE(v) \gtrsim -\log(v(0)) + v(0) -1.
    \end{equation}
    In particular, $\cE$ has a definite lower bound.
\end{cor}

The minimizing function $v$ is a priori only convex, so in order to show that $v$ solves ~\eqref{eqn:FreeBdyEqnP}, we must show that it is an Alexandrov solution and then apply the regularity theory due to Caffarelli. 
\begin{prop}\label{prop:minimizingOblique}
    The function $v \in \cC^+(P)$ minimizing $\cE$ solves equation~\eqref{eqn:FreeBdyEqnP} and satisfies $v \in C^\infty(P) \cap C^{1,\epsilon}(\ol{P})$.
\end{prop}

\begin{proof}
    Let $\hat{v} \in \cC^+(P)$ and $\hat{u} = \hat{v}^*$. 
    Since $v$ minimizes $\cE$, we know that 
    \[
    -\log I(v) + J^{-\frac{1}{n+1+\alpha}}(v) \leq -\log I(\hat{v}) + J^{-\frac{1}{n+1+\alpha}}(\hat{v}). 
    \]
    Let $u = v^*$ for $u$ the minimizer of $\cE$, and define two probability measures
    \[
    d\mu := \frac{(-u)^\beta h_\Sigma(\tfrac{x}{-u})\chi_{\Omega}}{\int_\Omega (-u)^\beta h_\Sigma(\tfrac{x}{-u})\,dx}dx\qquad \text{and}\qquad  d\nu := \frac{ v^{-(n+2+\alpha)} h_P(y)}{\int_Pv^{-(n+2+\alpha)}  h_P(y) \, dy}dy. 
    \]
    We need to show 
    \begin{equation}\label{eqn:OTIneq}
    \int_{\bR^n}u\, d\mu + \int_P v\, d\nu \leq \int_{\bR^n}\hat{u}\, d\mu + \int_P \hat{v}\, d\nu. 
    \end{equation}

    Let $v_t = v + t(\delta v)$ and consider 
    \begin{equation}\label{eqn:OTfromEnergy}
    \begin{aligned}
    0 &= \frac{d}{dt}\bigg\vert_{t = 0}\cE(v_t)\\
    &= \frac{1}{I(u)}\int_{\Omega} (-u)^\beta h_\Sigma(\tfrac{x}{-u(x)})(\delta u)\, dx + \frac{1}{J(v)^{1+\frac{1}{n+1+\alpha}}}\int_Pv^{-(n+2+\alpha)}  h_P(y) (\delta v)\, dy.
    \end{aligned}
    \end{equation}
    The variation $\delta u$ can be computed as  $\delta u = \frac{d}{dt}\big\vert_{t = 0}u_t^*$.
    By the Legendre transform,
    \[
    u_t(x) = \sup_{y \in P} \la x,y \rg - v(y) - t\delta v(y)
    \]
    is jointly convex in $(x,t)$. Therefore, we have the inequality
    \begin{equation}\label{eqn:deltaUIneq}
    \delta u = \pt u_t \vert_{t=0} \leq \int_0^1 \pt u_t\, dt = u_1 - u_0.
    \end{equation}
If we consider the path $v_t = v + t$, we see that both $\delta v = 1$ and $\delta u = -1$ are constants. 
    Therefore, equation~\eqref{eqn:OTfromEnergy} becomes
    \[
    \frac{\int_\Omega (-u)^\beta h_\Sigma(\tfrac{x}{-u})\, dx}{I(u)} = \frac{\int_Pv^{-(n+2+\alpha)}  h_P(y) \, dy}{J^{1 + \frac{1}{n+1+\alpha}}(v)}.
    \]
    Dividing ~\eqref{eqn:OTfromEnergy} by this constant shows that the minimizing property is
    \begin{equation}\label{eqn:variationalInequalityMinimizer}
         0 = \int_{\bR^n} (\delta u)\, d\mu + \int_P (\delta v)\, d\nu.
    \end{equation}
Consider the family $v_t = v + t(\hat{v} - v)$ so that $\delta v = (\hat{v} - v)$.
By \eqref{eqn:deltaUIneq}, we have $\delta u \leq \hat{u} - u$. 
Combining this with ~\eqref{eqn:variationalInequalityMinimizer}, we obtain
    \[
    0 \leq \int_{\bR^n}(\hat{u} - u)\, d\mu + \int_P (\hat{v} - v) \, d\nu,
    \]
    which is the desired inequality~\eqref{eqn:OTIneq}.

    The results of Gangbo-McCann \cite{Gangbo-McCann} tell us that since $\nabla u$ solves the optimal transport equation from $d\mu$ to $d\nu$, $u$ solves equation~\eqref{eqn:freeBdyEqnOmega} in the Alexandrov sense.
    Since $h_P$ is a doubling measure on $P$, we can apply the Caffarelli interior regularity theory \cite{Caffarelli} to upgrade the weak Alexandrov solution to a $C^{1,\eps}$ solution in $P$.
    Moreover, from Jhaveri-Savin \cite[Theorem 1.1]{Jhaveri-Savin}, we know that there is some $\epsilon > 0$ such that $v$ is $C^{1,\epsilon}$ at the boundary of $P$ for some $\epsilon > 0$ and $u \in C^{1,\epsilon}(\bR^n)$.
\end{proof}

\begin{proof}[Proof of Theorem~\ref{thm:StrongObliqueness}]
Theorem~\ref{thm:StrongObliqueness} follows from scaling and homogenizing $v$, as discussed in Remark~\ref{rem:rescaling}.
\end{proof}

\begin{remark}
    Without the assumption of strong obliqueness, one can construct explicit families of functions for which the energy functional $\mathcal{E}$ is not bounded from below.
\end{remark}

\section{Strongly partially oblique cones}\label{sec:Non-strict}

We now examine the case where the target cone is not a strict cone. 
Let us assume that $\Sigma = \Sigma^k \times \bR^{n-k}$ where $\Sigma^k\subset \mathbb R^k$ is a bounded convex set. 
We introduce the following notation arising from the splitting induced by $\Sigma$.
\begin{itemize}
    \item Define $\bA''$ to be the $k$-dimensional subspace containing $\Sigma^\circ$ and $\bA'$ to be its annihilator in the dual space.
    \item The free boundary domain is $\Omega = \{ u + \phi_{\Sigma^\circ} < 0\}$ for $u = v^*$. 
    We denote by $\Omega'$ the slice $\Omega \cap \bA'$. 
\end{itemize}

Since $\phi_{\Sigma^\circ}$ is independent of the $x'$ directions, the functional $I$ is invariant under translations in $x'$.
However, $J$ is not.
This is similar to the situation which appeared in \cite{TristanBenjy,TristanFreidYau}, where $\phi_{\Sigma^{\circ}} \equiv 0$. 
The strategy in this context combines the approach normalizing $J$ over the translations in the $\bA'$ directions together with the strategy in the previous section deducing the compactness in the $\bA''$ directions from the strong obliqueness condition.

Motivated by this invariance, we define the space of {\em normalized} convex functions to be
\[
\cC'(P) := \left\{v \in \cC^+(P) : \int_P \langle y, x'\rangle \frac{h_P}{v^{n+2+\alpha}}\, dy = 0  \text{ for all } x' \in \bA'\right\}.
\]

\begin{defn}\label{defn: vanishingOrder}
    We say that $h_P$ has vanishing order at most $\mathtt{v}$ if for all $y_0 \in \p P$, there is some cone $\mathtt{C}$, and $\delta > 0$, such that
    \[
    h_P(y) \gtrsim d(y,y_0)^{\mathtt{v}}\quad \text{for all } y \in \{y_0 + \mathtt{C}\} \cap B_\delta(y_0).
    \]
\end{defn}

We first show that for every function $v \in \cC^+(P)$, there is an affine function $\ell$ such that $v - \ell \in \cC'(P)$.
\begin{lem}\label{lem:normalization}
    Let $\Omega' =\{w < 0\}\cap \bA'$ and suppose that $h_P$ has vanishing order at most $1 + \alpha$. 
    For any $v \in \cC^+(P)$, the map $\Omega' \ni x \mapsto f(x') = J(v - \la x', \cdot\rg)$ is a strictly convex function defined on $\Omega'$ which diverges to $\infty$ as $x' \to \p \Omega'$.
    Furthermore, $f$ is minimized at the unique point $x_0' \in \Omega'$ such that 
    \[
    \int_P \langle {y}, x'\rangle \frac{h_P(y) }{(v - \la x_0',y\rg)^{n+2+\alpha}}\,dy = 0
    \]
    for all $x' \in \bA'$.
\end{lem}
\begin{proof}
    Let $v_t = v - t\la x', \cdot \rg$.
    Since, $x' \in \Omega'$, we have $v_t > 0$ in $P$.
    Let $\Phi(t) := J(v_t)$. 
    We compute
    \[
    \dot{\Phi} = \int_P \frac{h_P}{v_t^{n+2+\alpha}}\la y, x'\rg\,dy\qquad \text{and}\qquad \ddot{\Phi} = (n+2+\alpha)\int_P \frac{h_P}{v_t^{n+3+\alpha}}\la  y,x'\rg^2\,dy.
    \]
    For $x_0' \in \Omega'$, define $v_{x_0'} = v - \la x_0', \cdot \rg$, and write  $u_{x_0}(x) = u(x+x_0)$ and $w_{x_0} = u_{x_0} + \phi_{\Sigma^\circ}$. 
    By duality, $\{w_{x_0} < 0\} = \Omega- x_0$.
    Recall that $h_P(y) \gtrsim d_{\p P}^{1+\alpha}$. 
    Suppose that $v_{x_0'}$ vanishes at some $y_0 \in \ol{P}$. 
    There is some cone $\mathtt{C}$ with non-empty interior such that $K_{y_0} = \{y_0 + \mathtt{C}\} \cap B_\delta(y_0) \subset P$.
    We can estimate
    \[
    J(v_{x_0}) = \int_P v^{-(n+1+\alpha)}h_P(y)\,dy \gtrsim \int_{K_{y_0}} d(y,y_0)^{-n}\,dy = \infty.
    \]
    Therefore, $f$ diverges to infinity as $x'$ approaches $\p \Omega'$, proving the result.
\end{proof}

The following key estimate controls the $I$ functional.  The idea is that $\Omega$ can be viewed as a fibration over $\Omega'$, and by strong partial obliqueness, the fibers are controlled by $|u(x')|$.
This follows from a fiberwise version of Lemma~\ref{lem:RoundnessOfOmega} at the infimum of $u$ over $\Omega'$.
\begin{lem}\label{lem:IBoundNonStrictCones}
    There is a lower bound on $-\log I(u)$ given by 
    \[
    -\log I(u) \gtrsim -\log (|\Omega'|(v(0))^{k+1+\beta}) - 1.
    \]
\end{lem}
\begin{proof}
    We can estimate
    \begin{equation}\label{eqn:IestimateFromK}
    I(u) \leq |\Omega|[\sup_{\Omega} K(x,-w(x))] \lesssim |\Omega|[\sup_\Omega (-u(x))^{1 + \beta}] = |\Omega|(v(0))^{1+\beta}
    \end{equation}
    from Lemma~\ref{lem:IKernelEstimate}. 
    To complete the proof, we will show that 
    \[
    |\Omega| \lesssim 
    |\Omega'|(v(0))^k |
    \]
    using strong partial obliqueness.

    Fix the standard inner product on $\bR^n_y$, and take the inner product on $\bR^n_x$ induced by duality.
    Given $\bA''\subset \bR^n_y$, we can define $\tilde{\bA}'$ as the orthogonal subspace.
    Let $\tilde{\bA}''$ be the annihilator of $\tilde{\bA}'$ under the duality pairing.  Then $\tilde{\bA}''$ is orthogonal to $\bA'$. 
    Let
    \[
    \Pi': \bR^n_x \rightarrow \bA'.
    \]
    We write $x=x'+x''$ as the orthogonal decomposition, with $x' \in \bA'$ and $x'' \in \tilde{\bA}''$ denoting the orthogonal projection.
    Finally, denote by $\tilde{\bS}^k \subset \tilde{\bA}''$ the collection of all unit vectors in $\tilde{\bA}''$.

    We define
    \[
        \frac{1}{R} = \inf_{\tilde{\theta} \in\tilde{\bS}^k}\phi_{\Sigma^\circ}(\tilde{\theta}) - \phi_{P\cap \bA''}(-\tilde{\theta}) \qquad \text{and}\qquad \frac{1}{r} = \sup_{\theta \in {\bS}^n} \phi_{\Sigma^\circ}({\theta}) + \phi_{P}({\theta}).
    \]
    The strong partial obliqueness condition states that $\tfrac{1}{R} > 0$.
    Let $\Omega_{x'} = \{x \in \Omega : \Pi'(x) = x'\}$. 
    Since $\nabla u \in P$, the definition of $r$ shows 
    \[
    w(x' + x'') \leq u(x') + \frac{1}{r}|x''|,
    \]     
    which implies we have the containment
    \[
    B''_{r|u(x')|}(x') \subset \Omega_{x'}\quad \text{where}\quad  B''_{\rho}(x') =\{x=x' + x'' \in \mathbb R^n_x : |x''| \leq \rho \}.
    \]
    Since $\phi_{\Sigma^\circ}\big|_{\bA'} \equiv 0$, the boundary data imply that $u\big\vert_{\p \Omega'} = 0$.
    In particular, there exists some $x_0' \in \mr{int}(\Omega')$ at which $u\big|_{\Omega'}$ is minimized. 
    We may therefore choose $y \in \p u(x_0') \cap \bA''\subset P\cap \bA''$.
    Since $u$ sits above the supporting hyperplane defined by $y$, we have 
\begin{align*}
        u(x' + x'') &\geq  u(x_0') + \la x' - x'_0,y\rg +\la  x'',y\rg\\
        & \geq u(x_0') - \phi_{P \cap \bA''}(-x'')
    \end{align*}
    using $\la - x'', y\rg \leq \phi_{P \cap \bA''}(-x'')$ in the last line by our choice of $y \in P\cap \bA''$. 
    Adding $\phi_{\Sigma^\circ}(x)$ and using that $\phi_{\Sigma^\circ}$ is independent of $x'$, we get 
    \begin{equation}\label{eqn:wLowerEstimePSO}
    w(x' + x'') \geq u(x_0') + \phi_{\Sigma^\circ}(x'') - \phi_{P \cap \bA''}(-x'') \geq u(x_0')+\frac{1}{R}|x''|
    \end{equation}
    using strong partial obliqueness.
    Therefore, we see 
    \begin{equation}\label{eqn:OmegaFiberEst}
    |\Omega| = \int_{\Omega'}|\Omega_{x'}|\,dx' \leq R^k |\Omega'||u(x'_0)|^k.
    \end{equation}
    
    To complete the proof, we show that $u(x'_0) \sim \inf_\Omega u$. 
    From equation~\eqref{eqn:wLowerEstimePSO}, $\Omega_{x'_0}$ will be trapped by the inequalities on $w(x_0' + x'')$
    \[
    u(x_0') + \frac{1}{R}|x''| \leq w(x_0' + x'') \leq u(x_0') + \frac{1}{r}|x''| 
    \]
    analogously to equation~\eqref{eqn:wMinAt0} in the strongly oblique setting.
    From equation~\eqref{eqn:wLowerEstimePSO}, for any $x \in \Omega'$, we know that $|x''| \leq R|u(x_0')|$, or equivalently, $-|x''| \geq Ru(x_0')$.
    Letting $\gamma = \sup_{\tilde{\theta} \in \tilde{\bS}^k} \phi_{P \cap \bA''}(\tilde{\theta})$ and combining the above estimates shows 
    \begin{equation}\label{eqn:infOmegasiminfOmega'}
    -v(0) = \inf_\Omega u = u(x' + x'') \geq u(x_0') - \phi_{P \cap \bA''}(-x'') \geq u(x_0') - \gamma|x''| \geq (1 + \gamma R)u(x_0').
    \end{equation}
    Combining inequalities~\eqref{eqn:IestimateFromK}, \eqref{eqn:OmegaFiberEst}, and~\eqref{eqn:infOmegasiminfOmega'} completes the proof.
    
\end{proof}

We now show the requisite bound on $J$ that, in combination with Lemma~\ref{lem:IBoundNonStrictCones}, will show that $\cE$ is uniformly lower bounded on normalized functions. 

\begin{lem}\label{lem:v(0)infv}
    Let $v \in \cC^+(P)$ be such that $\Omega'$ is in John's position. 
    Then, there exists a uniform $c > 0$ depending on $(P,\Sigma)$ such that $cv(0) \leq \inf_P v$.
\end{lem}
\begin{proof}
    Consider $u\vert_{\Omega'}$.
    Applying \cite[Lemma 3.3]{TristanBenjy} shows that $\frac{|\inf_{\Omega'}u|}{n-k+2} \leq |u(0)|$. 
    To complete the proof, we recall that strong partial obliqueness shows $\inf_\Omega u(x) \sim \inf_{\Omega'}u(x')$. 
\end{proof}

\begin{lem}\label{lem:JBoundNonStrictCones}
Let $u =v^*$ be such that $\Omega' = \{u + \phi_{\Sigma^\circ} < 0 \} \cap \bA'$ is in John's position.
Then, we have
\[
J(v) \lesssim  |v(0)|^{-(k+1+\alpha)}|\Omega'|^{-1}.
\]
\end{lem}
\begin{proof}
    Using $h_P < C$, we apply the coarea formula to estimate
    \begin{align*}
    J(v) &= \frac{1}{n+1+\alpha}\int_P \frac{h_P(y) }{v^{n+1+\alpha}} \,dy\\
    &\lesssim \frac{1}{n+1+\alpha}\int_P \frac{dy}{v^{n+1+\alpha}} = \int_{\inf_P v}^\infty t^{-(n+2+\alpha)}|P \cap \{v < t\}|\,dt.
    \end{align*}
    We claim that we have the containment
    \[
    P \cap \{v < t\} \subset P \cap \bA'' \times (t\Omega')^\circ.
    \]
    From the definition of the Legendre transform, we know that at points $y$ where $v(y) < t$, for all $x\in \Omega$, we have
    \[
    \la x, y \rg \leq t + u(x) \implies \la x, y \rg \leq t.
    \]
    Specializing to $x' \in \Omega'$ shows that for any fixed $y''$, we have 
    \[
    \{v < t\} \cap \{y''\} \times \bR^{n-k} \subset  \{y''\} \times t(\Omega')^\circ.
    \]
    Integrating this over $y'' \in P \cap \bA''$ provides the upper bound
    \[
    \{v < t\} \lesssim t^{n-k}|(\Omega')^\circ|.
    \]
    
    Combining the upper bound on $J$ with the volume estimates of the level sets $P \cap \{ v < t\}$, we see that 
    \[
    J(v) \lesssim \int_{\inf_P v}^\infty t^{-(n+2+\alpha)}t^{n-k}|(\Omega')^\circ|\, dt = C|(\Omega')^\circ||\inf_{P}v|^{-(k+1+\alpha)}.
    \]
    Finally, since $\Omega'$ is in John's position, we know that $|(\Omega')^\circ| \sim |\Omega'|^{-1}$ and $v(0) \sim \inf_P v$ from Lemma~\ref{lem:v(0)infv}, completing the proof. 
\end{proof}

\begin{cor}\label{cor:ELowerBound}
    There exists some $C$ such that for any $v \in \cC'(P)$, we have $\cE(v) \geq - C$.
    Furthermore, for $v$ such that $\cE(v) < A$, there exists $0 < \delta_* \leq D^*$ such that 
    \[
    \delta_* \leq| \inf_{\Omega'} u| \lesssim |\Omega'|^{\frac{1}{n-k}}\lesssim \mr{diam}(\Omega') \leq D^*.
    \]
\end{cor}
\begin{proof}
    Let $\tilde{v} = v - \la x', \cdot \rg$ so that $\tilde{\Omega}'$ is in John's position.
    Combining the estimates for $I$ and $J$ from Lemmas~\ref{lem:IBoundNonStrictCones} and~\ref{lem:JBoundNonStrictCones}, we see that 
    \begin{align*}
    \cE(v) \geq \cE(\tilde{v}) &\gtrsim -\log(|\tilde{\Omega}'||\inf_{\tilde{\Omega}'}\tilde{u}|^{k+1+\beta}) + |\tilde{\Omega}'|^{\frac{1}{n+1+\alpha}} |\tilde{u}(0)|^{\frac{k+1+\alpha}{n+1+\alpha}}- 1\\
    &\gtrsim -\log(|\tilde{\Omega}'|(\tilde{v}(0))^{k+1+\beta}) + |\tilde{\Omega}'|^{\frac{1}{n+1+\alpha}} (\tilde{v}(0))^{\frac{k+1+\alpha}{n+1+\alpha}}- 1\\
    &\gtrsim -\log(|\Omega'|(v(0))^{k+1+\beta}) + |\Omega'|^{\frac{1}{n+1+\alpha}} (v(0))^{\frac{k+1+\alpha}{n+1+\alpha}}- 1.
    \end{align*}
    In particular, we used that $|\inf_{\Omega} w| \lesssim |\inf_{\Omega'} u|$ from the strong partial obliqueness property and the positivity $\phi_{\Sigma^\circ} > 0$. Furthermore, from the ABP estimate and $\nabla u(\Omega) \subset P$, we have 
    \[
    |u(0)|^{n-k} \lesssim  |\Omega'|.
    \] 

    The energy upper bound $\cE(v) \leq A$ implies that the product $|\Omega'||\inf_{\Omega'}u|^{k+1+\beta}$ has definite lower bound depending only on $A$.

    Suppose we have a degenerating sequence $\cE(v_i) \to -\infty$.
    Then, we must have $|\Omega_i'||\inf_{\Omega_i'}u_i|^{k+1+\beta} \to \infty$ since $J >0$. 
    We rule this out by showing that any such degeneration must actually have large energy.  
    First, we show $v(0)$ is bounded above, depending only on an upper bound $\mathcal{E}(v) \leq A$.
    Suppose to the contrary that $\mathcal{E}(v) \leq A$ but that $v(0) > 1$ is large. 
    Then, we have 
    \[
    A \geq \cE(v) \geq -\log(|\Omega'|^{\frac{n+1+\beta}{n-k}}) +  c_2|\Omega'|^\frac{1}{n+1+\alpha} - c
    \]
    and the ABP estimate forces $|\Omega'|$ to be large as well.

    Next, we show that $|\Omega'|$ is bounded above, and $v(0)$ is bounded below provided $\mathcal{E}(v) \leq A$.  Let $t = |\Omega'|(v(0))^{k+1+\beta}$, so the lower bound is 
    \begin{align*}
        A \geq  \cE(v) &\geq -\log(t) + c_2 (|\Omega'|(v(0))^{k+1+\beta})^\frac{1}{n+1+\alpha} (v(0))^{\frac{k+1+\alpha}{n+1+\alpha} - \frac{k+1+\beta}{n+1+\alpha}}- c \\
        &= -\log(t) + c_2 t^\frac{1}{n+1+\alpha}   (v(0))^{ \frac{\alpha-\beta}{n+1+\alpha}} - c.
    \end{align*}
    Since $\beta >\alpha$ and $v(0)$ is bounded from above, one easily sees from this estimate that $t$ is bounded above and below, and $v(0)$ is bounded from below depending only on $A$ and universal constants.
    
\end{proof}

We now show that the boundary condition is naturally occurring along minimizing sequences. 
\begin{lem}\label{lem:boundaryNormalization}
    For any $v \in \cC'(P)$, there exists $\tilde{v} \in \cC'(P)$ such that 
    \begin{itemize}
        \item[(i)] $I(v) = I(\tilde{v})$,
        \item[(ii)] $J(\tilde{v}) \geq J(v)$, and 
        \item[(iii)] $\tilde{v}(y) \geq \phi_{\Omega'}(y)$ with equality along $y \in \p P$.
    \end{itemize}
\end{lem}
\begin{proof}
    This proof follows as \cite[Lemma 4]{TristanFreidYau} using Lemma~\ref{lem:normalization} in place of \cite[Proposition 3.1]{TristanFreidYau} and only translating in the prime direction. 
    
    Let $\tilde{v}(y) = \sup_{x \in \ol{\Omega}}(\la x, y\rg - u(x)) + \la y , x_0'\rg$ where $x_0'$ is chosen from Lemma~\ref{lem:normalization} so that $\tilde{v} \in \cC'(P)$. 
    Then, we can define $\tilde{u}(x) = \sup_{y \in P} (\la x, y \rg - \tilde{v}(y))$, which gives the comparisons
    \begin{align*}
        \tilde{v}(y) &\leq v(y) + \la x_0', y\rg ,\\
        \tilde{u}(x) & \geq u(x - x_0').
    \end{align*}
    Since $\tilde{v}$ is the Legendre transform of the function $\hat{u}(x) + \mathbf{1}_{\Omega + x_0'}$, the Legendre duality $\hat{u}^{**} = \hat{u}$ tells us that for $x - x_0' \in \Omega$, we have $u(x-x_0') = \tilde{u}(x)$, so in particular, we know that $I(v) = I(\tilde{v})$, proving point $(i)$. 

    From Lemma~\ref{lem:normalization}, we know that $J(v) \leq J(v - \la x_0',\cdot \rg)$ and 
    \[
    \tilde{v}(y) = \sup_{x \in \Omega}(\la x,y\rg - u(x)) + \la y , x_0'\rg \leq \sup_{x \in \bR^n}(\la x,y\rg - u(x)) + \la y, x_0'\rg = v(y) + \la y,x_0'\rg,
    \]
    showing that 
    \[
    J(v) \leq J(v - \la x_0', \cdot \rg) \leq J(\tilde{v}),
    \]
    proving point $(ii)$. 

    Finally, we prove point $(iii)$.
    We first show that $v$ lies above $\phi_{\Omega'}$. 
    This follows by restricting the supremum in the Legendre transform from $\Omega$ to $\Omega'$
    \[
    \tilde{v}(y) \geq \sup_{x' \in \Omega'}(\la x', y\rg -u(0,x'-x_0')) = \sup_{x' \in \Omega'} \la x', y'\rg = \phi_{\Omega'}(y'),
    \]
    where we used that $u(0,\cdot) = 0$ on $\p \Omega'$ and $u(0,x') \geq 0$ outside of $\Omega'$. We now show that we achieve equality at the boundary. If at some boundary point $y \in \p P$ the supporting hyperplane of $\hat{u}$ achieving $\tilde{v}(y)$ met the graph of $\hat{u}$ only at interior points, then $y$ would be interior to $P$, contradicting the assumption. 
\end{proof}
From this lemma, we realize that the boundary condition is naturally occurring on minimizing sequences.
Going forward, we can therefore assume that we achieve the desired boundary criterion.
We now prove an effective version of the normalization result. 

\begin{prop}\label{prop:effective}
    Let $v \in \cC'(P)$ satisfy the conditions of Lemma~\ref{lem:boundaryNormalization} and satisfy $\cE(v) \leq A$.
    Then, there exists $c > 0$ such that $c \inf_P v \geq v(0)$.

\end{prop}
\begin{proof}
    For the sake of contradiction, suppose there exists a sequence $v_k$ satisfying $\cE(v_k) < A$ such that 
    \[
    \inf_P v_k = v_k(y_k) \leq k^{-1}v_k(0).
    \]
    From Corollary~\ref{cor:ELowerBound}, we have $\delta_*, D^*$ such that 
    \[
        0 < \delta_* \leq v_k(0) \lesssim \mr{diam}(\Omega'_k) \leq D^*.
    \]
    As in the proof of Lemma~\ref{lem:IBoundNonStrictCones}, we construct a basis $x''_i$ completing $\bA'$ to a full basis of $\bR^n_x$ compatible with the inner product structure induced by duality and giving a well-defined projection operator.
    By the assumption that $v_k(y) \geq \phi_{\Omega'_k}(y')$ with equality on $\p P$, we know that $\Pi'(\nabla v_k(P)) = \Omega'_k$.
    Therefore, we have uniform gradient bounds on $v_k$ in the prime direction, which, in combination with the strong partial obliqueness property, yield uniform gradient estimates for $\nabla v_k$. 
    Thus, we may take a uniform limit $v_k \to v_\infty$ and $y_k \to y_\infty$ which satisfies $v_k(y_k) = 0$. 
    From Lemma~\ref{lem:boundaryNormalization}, we know that $\int_P \frac{h_P(y) }{v_\infty^{n+1+\alpha}}\,dy = + \infty$, so from Fatou's lemma, 
    \[
    +\infty = \int_P \frac{h_P(y) }{v_\infty^{n+1+\alpha}}\,dy \leq \liminf_{k \to \infty} \int_P \frac{h_P(y) }{v_k^{n+1+\alpha}}\,dy.
    \]
    Let $\tilde{v}_k = v_k - \la x_k',y\rg$, where $x_k'$ is chosen so that $\Omega'_k - x_k'$ is in John's position. 
    For $u_k = v^*$, strong partial obliqueness (see equation~\eqref{eqn:infOmegasiminfOmega'}) from Lemma~\ref{lem:IBoundNonStrictCones} shows that $\inf_{\Omega} u_k \sim u_k(0)$ which by Legendre duality shows that $v_k(0) \sim \inf_P v_k$. 
    Applying Lemmas~\ref{lem:normalization},  \ref{lem:v(0)infv}, and~\ref{lem:JBoundNonStrictCones}, we see there is a uniform $C$ such that 
    \[
    \int_P \frac{h_P(y) }{\tilde{v}_k^{n+1+\alpha}}\,dy < C,
    \]
    a contradiction.
\end{proof}
We can now take a minimizing sequence in $\cC'(P)$ whose limit will solve equation~\eqref{eqn:FreeBdyEqnP}.
\begin{prop}
    There is a minimizer $v \in \cC'(P)$ of $\cE$.
\end{prop}
\begin{proof}
    Let $v_k$ be a minimizing sequence, and from Lemma~\ref{lem:boundaryNormalization}, we assume that this sequence satisfies the boundary condition in the prime directions.
    Proposition~\ref{prop:effective} shows that $v_k(0)$ and $\inf_P v_k$ are uniformly equivalent, and Corollary~\ref{cor:ELowerBound} gives $\delta_* \leq \inf_P v_k \leq v_k(0) \leq D^*$.
    Therefore, we have uniform $C^1$ bounds on $u_k$, so we can extract a uniform limit $u_k \to u_\infty$ which will satisfy $u_\infty = v_\infty^*$. 
    By uniform convergence, $v_\infty \in \cC'(P)$ and $\cE(v_k) \to \cE(v_\infty) = \inf_{v \in\cC'(P)} \cE(v)$ as desired.
\end{proof}
We define $v \in \cC'(P)$ to be a minimizer of $\cE$ and $u = v^*$. 
Since $v$ is not a priori $C^2$, we show that $v$ solves equation~\eqref{eqn:FreeBdyEqnP} in the Alexandrov sense by showing that $\nabla v$ solves an appropriate optimal transport solution equivalent to the Monge-Amp\`ere equation~\eqref{eqn:FreeBdyEqnP} by \cite{Gangbo-McCann}. 
\begin{prop}
    The minimizing function $v$ solves equation~\eqref{eqn:FreeBdyEqnP} and satisfies $v \in C^{1,\epsilon}(\ol{P})$. 
\end{prop}
\begin{proof}
    This proof follows closely to that of Proposition~\ref{prop:minimizingOblique}, with the complication that we can only deduce the positivity of $\frac{d}{dt}\cE$ at families in $\cC'(P)$. 
    To do this, we take an arbitrary family in $\cC^+(P)$ and normalize it, and use the properties of the probability measure to show the desired optimal transport inequality still holds. 
    
    We define two probability measures
    \[
    d\mu := \frac{(-u)^{ \beta}h_\Sigma(\tfrac{x}{-u})\chi_{\Omega}}{\int_\Omega (-u)^{ \beta}h_\Sigma(\tfrac{x}{-u})\,dx}dx\qquad \text{and}\qquad  d\nu := \frac{v^{-(n+2+\alpha)}h_P }{\int_P v^{-(n+2+\alpha)}h_P \, dy}dy. 
    \]
    The desired optimal transport inequality we must show is
    \begin{equation}\label{eqn:OTforNonStrict}
    \int_{\bR^n}u\, d\mu + \int_P v\, d\nu \leq \int_{\bR^n}\hat{u}\, d\mu + \int_P \hat{v}\, d\nu ,
    \end{equation}
    where, since $d\mu$ and $d\nu$ are probability measures, the right-hand side is invariant under $\hat{v} \mapsto \hat{v}+t$. 
    Using Lemma~\ref{lem:boundaryNormalization}, we replace $\hat{v}$ with $\tilde{v}$ satisfying $J(\tilde{v}) = J(v)$, which does not change the right-hand side of equation~\eqref{eqn:OTforNonStrict} since $d\nu$ is centered at the origin. 

    By construction, both $v_0$ and $v_1$ are normalized and in $\cC'(P)$.
    However, the interpolates $v_t$ need not be. 
    Therefore, we define $\psi_s = v_s + \la x_s', y\rg \in \cC'(P)$ which normalizes the family. 
    Let 
    \[
    \gamma_s(t) := v + \frac{t}{s}(\psi_s - v),
    \]
    which interpolates from $v$ to $\psi_s$ for $t \in [0,s]$. 
    Since $v$ is a minimizer over $\cC'(P)$, we know that 
    \begin{equation}\label{eqn:vMinimizingEC'}
    \frac{d}{dt}\bigg\vert_{t = t_*(s)} \cE (\gamma_s(t)) \geq 0 
    \end{equation}
    for some $t_*(s) \in [0,s]$ by the mean value theorem, since $\frac{\cE(\psi_s) - \cE(v)}{s} > 0$.
    As in the proof of Proposition~\ref{prop:minimizingOblique}, we can use the test family where $\delta v = 1$ and normalize the family to express the inequality~\eqref{eqn:vMinimizingEC'} using the measures $d\mu$ and $d\nu$ as
    \[
    0 \leq \int_{\bR^n} \pt u_{s,t}(x)\big|_{t=t_*(s)}\,d\mu + \int_P \left(\tilde v - v + \frac{\langle x_s',y\rangle}{s}\right)\,d\nu.
    \]
    
    The points $x_s'$ are chosen in a differentiable way as seen by the implicit function theorem.
    Consider
    \[
    \Psi : [0,1] \times \bR^{n-k} \to \bR^{n-k},\qquad \Psi(t,x') := \int_P \vec{y'}\frac{h_P}{(v_t - \la x', y \rg )^{n+2+\alpha}}\, dy
    \]
    where $\vec{y'} = (y_{k+1}',\ldots, y_{n}')$ is the dual basis of $\bA'$, satisfying $\bA'' \oplus \mr{span}(y_i') = \bR^n_y$. 
    We are searching for $x'_t = \varphi(t)$ satisfying $\Psi(t, \varphi(t)) = 0$. 
    Therefore, we differentiate
    \[
    \frac{\p \Psi_i}{\p x'_j}(0,0) = (n+2+\alpha)\int_P \frac{h_P}{v^{n+3+\alpha}}y_i'y_j'\, dy,
    \]
    which shows that the Jacobian,
    \[
    J_\Psi(\xi', \xi') = (n+2+\alpha)\int_P \frac{h_P}{v^{n+3+\alpha}}\la \xi', y\rg^2\, dy,
    \]
    is positive-definite, and therefore invertible. 
    Moreover, since the integrand is $C^\infty$ in $(x',t)$, the points $x_t'$ vary smoothly with $x_0' = 0$. 
    In particular, $\lim_{s \to 0} \frac{x_s'}{s} =: x'_\infty$ exists. 
    
    The Legendre transform $u_s=v_s^*$ satisfies
    \[
    \frac{d u_s}{ds}\bigg|_{s=0} \leq \tilde{u}-u,
    \]
    since $u_s$ is a jointly convex function in $(x,s)$ which interpolates from $u$ to $\hat{u}$ in time $1$, so
    \[
    u_s \leq (1-s)u +s \hat{u}.
    \]
    Therefore, taking the limit $s \to 0$ shows 
    \[
    \lim_{s \to 0}\int_{\bR^n} \pt u_{s,t}(x)\big|_{t=t_*(s)}\,d\mu \leq \int_{\bR^n} (\tilde{u} - u)\,d\mu.
    \]
    Finally, since $s^{-1}x_s'$ exists, we can take the limit as $s \to 0$,
    \[
    \lim_{s\to 0}\int_P\frac{h_P}{(\gamma_s(t_*))^{n+2+\alpha}}\left((\tilde{v} - v) + \frac{\la x_s',y\rg}{s}\right)\,dy = \int_{P} (\tilde{v} - v)\,d\nu + \lim_{s \to 0}\int_P \frac{\la x_s', y\rg}{s}\,d\nu
    \]
    and since $d\nu$ has barycenter at the origin, the last term vanishes in the limit
    \[
    \lim_{s\to 0}\int_P \frac{\la x_s',y'\rg}{s}\, d\nu = \int_P \la x_\infty',y\rg\,d\nu= 0.
    \]
    Combining the above two facts with the minimizing property of $v$ from inequality~\eqref{eqn:vMinimizingEC'} shows the desired optimal transport inequality.
    
    As in the proof of Proposition~\ref{prop:minimizingOblique}, we apply the results of Gangbo-McCann \cite{Gangbo-McCann}, Caffarelli \cite{Caffarelli}, and Jhaveri-Savin \cite{Jhaveri-Savin} to deduce that $u \in C^{1,\epsilon}(\bR^n)$ solves equation~\eqref{eqn:freeBdyEqnOmega} and $v \in C^{1,\epsilon}(\ol{P})$ solves equation~\eqref{eqn:FreeBdyEqnP}. 
\end{proof}

\begin{proof}[Proof of Theorem~\ref{thm:PartiallyStrongObliqueness}]
Theorem~\ref{thm:PartiallyStrongObliqueness} follows from scaling and homogenizing $v$, as discussed in Remark~\ref{rem:rescaling}.
\end{proof}

\medskip

\end{document}